\documentclass{article}
\usepackage{graphicx} 
\usepackage{amsmath}
\usepackage{amssymb}
\usepackage{amsthm}
\usepackage{amscd}
\usepackage{xypic}
\usepackage{delarray}
\usepackage{mathrsfs}
\usepackage{kotex}
\usepackage{mdframed}
\usepackage{lineno,hyperref}

\headsep=1cm \numberwithin{equation}{section}
\def\F{{\mathbb F}}

\def\Q{{\mathbb Q}}

\def\Z{{\mathbb Z}}
\newtheorem{theorem}{Theorem}[section]
\newtheorem{definition}[theorem]{Definition}
\newtheorem{lemma}[theorem]{Lemma}
\newtheorem{corollary}[theorem]{Corollary}
\newtheorem{remark}[theorem]{Remark}
\newtheorem{example}[theorem]{Example}
\newtheorem{proposition}[theorem]{Proposition}

\newtheorem{open problem}[theorem]{Open problem}

\title{Frobenius Alexander quandles, knot colorings, and isogeny classes of abelian varieties over finite fields}
\author{WonTae Hwang}
\date{}
\begin{document}
    
\maketitle

\begin{abstract}
We introduce the notion of Frobenius Alexander quandles associated with abelian varieties over finite fields, and study them from both knot-theoretic and arithmetic viewpoints. We determine exactly for which rational primes a knot admits a nonconstant coloring by these quandles, in terms of the resultant of its Alexander polynomial and the Frobenius characteristic polynomial of the abelian variety. This leads to the notion of persistent colorability and its characterization, as well as applications to fibered knots and knots of small genus. We also prove a rigidity result showing that the isomorphism class of a Frobenius Alexander quandle determines the similarity class of the Frobenius endomorphism on the corresponding torsion subgroup. As a consequence, one sufficiently large Frobenius Alexander quandle determines the characteristic polynomial of the Frobenius endomorphism, and hence, determines the isogeny class of the abelian variety over the given finite base field.
\end{abstract}




\section{Introduction}
Quandles were introduced independently by Joyce and Matveev in connection with knot theory \cite{Joy82,Mat82}, and finite Alexander quandles provide a classical source of knot-coloring invariants. The relation between Alexander quandle colorings and classical Alexander invariants was studied by Inoue \cite{Ino01}, who described quandle homomorphisms from knot quandles to Alexander quandles in terms of Alexander polynomials. A module-theoretic description of such quandle homomorphisms in terms of the Alexander module was later obtained by Inoue \cite{Ino10} and subsequently generalized by Traldi \cite{Tra18}. Since an Alexander quandle is determined by an abelian group together with one of its automorphisms, natural automorphisms coming from other areas of mathematics can be used to produce new families of coloring quandles.
\vskip 0.1in
Let $A$ be an abelian variety of dimension $g$ over a finite field
$\mathbb F_q$, where $q=p^a$ for some prime $p$ and an integer $a \geq 1$, and let $\pi_A$ be the Frobenius endomorphism of $A$. For every prime $\ell\neq p$, the action of Frobenius endomorphism on $A[\ell](\overline{\F}_q) \cong \F_{\ell}^{2g}$ defines the finite Alexander quandle $Q(A[\ell](\overline{\F}_q), \pi_{A,\ell})$, where $\pi_{A,\ell}$ is the restriction of $\pi_A$ to $A[\ell](\overline{\F}_q)$. We call this Alexander quandle the \emph{Frobenius Alexander quandle} associated with $A$ and $\ell.$ On the arithmetic side, Frobenius endomorphisms and their characteristic polynomials play a fundamental role in the theory of abelian varieties over finite fields, as shown in the work of Tate and Honda on isogeny classes \cite{Tate66,Hon68}. The construction considered here brings this arithmetic structure into the setting of Alexander quandles.
\vskip 0.1in
In this paper, we study these Frobenius Alexander quandles from two points of view. On the one hand, we use them as targets for knot colorings and determine when nonconstant colorings occur for infinitely many rational primes $\ell \ne p$. On the other hand, we study how much arithmetic information about $A$ is encoded by the abstract quandle $Q(A[\ell](\overline{\F}_q), \pi_{A,\ell}).$ In particular, we show that a single sufficiently large Frobenius Alexander quandle determines the Frobenius polynomial, whence, the isogeny class of $A$. To be more explicit, let $K$ be an oriented knot with normalized Alexander polynomial $\Delta_K(t) \in \Z[t]$, and let $\Sigma_A(K)$ be the set of rational primes $\ell\neq p$ for which $K$ admits a nonconstant coloring by the Frobenius Alexander quandle $Q(A[\ell](\overline{\F}_q),\pi_{A,\ell})$. If $P_A(t)$ denotes the characteristic polynomial of $\pi_A$ and if we let $R_{A,K}=\operatorname{Res}(\Delta_K(t),P_A(t)),$ then the following dichotomy is one of our main results of the paper.
\begin{theorem}[Theorem \ref{main thm A}]
    Let $A$ be an abelian variety over a finite field $\F_q$, and let $K$ be an oriented knot with the normalized Alexander polynomial $\Delta_K(t) \in \Z[t]$. Then we have
    $$\Sigma_A(K) =
\begin{cases}
\{\ell \in \mathcal{P} : \ell \neq p\}& \textrm{if}~R_{A,K}=0,\\
\{\ell \in \mathcal{P} : \ell \neq p~\textrm{and}~ \ell\mid R_{A,K}\} & \textrm{if}~R_{A,K}\neq0.
\end{cases}$$
Here, and throughout the paper, $\mathcal{P}$ denotes the set of all rational primes.
\end{theorem}
In view of the above theorem, we say that $K$ is \emph{persistently $A$-colorable} if the set $\Sigma_A(K)$ is infinite. Then we show that this is equivalent to $\gcd(\Delta_K(t),P_A(t))\neq1$ in $\mathbb Q[t]$, and also to the existence of a simple isogeny factor $B$ of $A$ such that $\sharp B(\F_q)=1$ and $m_B(t) \mid \Delta_K(t),$ where $m_B(t)$ is the minimal polynomial of the Frobenius endomorphism $\pi_B$ of $B$. Consequently, given a fixed abelian variety $A$ over a finite field $\F_q$, persistent $A$-colorability of an oriented knot $K$ can occur only for $q\in\{2,3,4\}$ (see Corollary \ref{main cor B}), and an oriented knot $K$ with monic normalized Alexander polynomial $\Delta_K(t) \in \Z[t]$, in particular, fibered knots, are never persistently $A$-colorable (see Corollary \ref{main cor D}). Note also that the condition $\sharp B(\mathbb F_q)=1$ connects our criterion with the classical study of abelian varieties over finite fields having only one rational point, as in the work of Madan and Pal \cite{MP77}. We further refine this criterion in terms of the genus of the knot. In particular, we obtain a complete characterization of persistent $A$-colorability for genus-one knots and an explicit classification in the genus-two case. As an application, among the twist knots $J_n$ with $n \geq 2$, it turns out that the stevedore knot $J_2 = 6_1$ is the unique knot which is persistently $A$-colorable for some abelian variety $A$ over a finite field. We also show that persistent $A$-colorability is preserved under connected sums of knots in the sense that for two oriented knots $K_1$ and $K_2,$ we have that $K_1 \sharp K_2$ is persistently $A$-colorable if and only if at least one of $K_1$ and $K_2$ is persistently $A$-colorable.
\vskip 0.1in
We now turn to the arithmetic rigidity of Frobenius Alexander quandles.
Using the classification of finite Alexander quandles \cite{Nel03}, we show that for finite-dimensional $\mathbb F_\ell$-vector spaces $V,W$ and automorphisms
$S \in \mathrm{GL}(V), T \in \mathrm{GL}(W)$, we have $Q(V,S) \cong Q(W,T)$ as quandles if and only if $(V,S)$ and $(W,T)$ are isomorphic as $\mathbb F_\ell[t,t^{-1}]$-modules, where $t$ acts on $V$ (resp.\ $W$) as $S$ (resp.\ $T$). Thus the abstract quandle $Q(A[\ell](\overline{\F}_q), \pi_{A,\ell})$ determines the similarity class of $\pi_{A,\ell}$, and hence, the reduction of $P_A(t)$ modulo $\ell$. This gives rise to the following isogeny criterion.
\begin{theorem}[Corollary \ref{main cor J}]
  Let $A$ and $B$ be abelian varieties of dimension $g$ over a finite field $\F_q$ where $q=p^a$ for some prime $p$ and an integer $a \geq 1.$ Then $A$ is $\F_q$-isogenous to $B$ if and only if $Q(A[\ell](\overline{\F}_q),\pi_{A,\ell}) \cong Q(B[\ell](\overline{\F}_q), \pi_{B,\ell})$ as quandles, for all but finitely many rational primes $\ell \ne p.$  
\end{theorem}
More effectively, if $\ell \ne p$ satisfies $\ell>2\binom{2g}{g}q^{g/2},$ then the isomorphism class of the single Frobenius Alexander quandle $Q(A[\ell](\overline{\F}_q), \pi_{A,\ell})$ determines the integral Frobenius polynomial $P_A(t)$. As a consequence, we see that, for every rational prime $\ell>2\binom{2g}{g}q^{g/2},$ if $Q(A[\ell](\overline{\F}_q),\pi_{A,\ell}) \cong Q(B[\ell](\overline{\F}_q),\pi_{B,\ell})$ as quandles, then $A$ is $\F_q$-isogenous to $B.$ In particular, one sufficiently large finite Frobenius Alexander quandle determines the $\mathbb F_q$-isogeny class of the given abelian variety (see Corollary \ref{main cor I}). 
\vskip 0.1in
The paper is organized as follows: In Section 2, we recall some necessary facts on the theory of (Alexander) quandles ($\S$2.1), and the theory of abelian varieties over finite fields ($\S$2.2). In Section 3, for a given abelian variety $A$ over a finite field $\F_q$ and an oriented knot $K,$ we study the Frobenius Alexander quandle colorings associated with $A$ and $K,$ and then we introduce the notion of persistent $A$-colorability of the knot $K$, together with some knot-theoretic properties. In Section 4, we establish the arithmetic rigidity of Frobenius Alexander quandles and prove the resulting isogeny criteria. Finally, combining the results on knot colorings with the rigidity of Frobenius Alexander quandles, we obtain an isogeny criterion from the Alexander quandles associated with the coloring spaces.

\section{Preliminaries}
    \subsection{Quandles and Alexander quandles}
In this section, we briefly recall the basic definitions in the theory of quandles, and fix our notation for our later use. Our main references are \cite{Joy82, EN15, Kam02, Nel03}.
\begin{definition}\cite{EN15}
    A \emph{quandle} is a nonempty set $Q$ equipped with a binary operation
$*$ satisfying $x * x=x$ for all $x \in Q$, and such that, for every $y\in Q$, the map $R_y \colon Q \rightarrow Q$ given by $R_y(x)=x * y$ is bijective, and such that the equality
$(x * y)* z = (x* z)*(y* z)$ holds for all $x,y,z\in Q$.
\end{definition}

\begin{definition}\cite{Joy82, EN15}\label{Alexander quand def}
Let $K$ be an oriented knot and let $D$ be an oriented diagram of $K$. The \emph{fundamental quandle} $Q(K)$ of the knot $K$ is the quandle generated by the arcs of $D$, subject to the quandle relation determined by each crossing of $D$. (Note that its quandle isomorphism class is independent of the choice of the oriented knot diagram $D$ of $K$.)    
\end{definition}

\begin{definition}\cite{Kam02}
Let $K$ be an oriented knot and let $D$ be an oriented diagram of $K$. Also, let $Q$ be a quandle. Then a $Q$-coloring of $D$ is an assignment of elements of $Q$ to the arcs of $D$ satisfying the corresponding quandle relation at every crossing. The set of $Q$-colorings $\operatorname{Col}_Q (K)$ is naturally identified with the set of quandle homomorphisms $\operatorname{Hom}_{\mathrm{Quand}}(Q(K),Q)$ from $Q(K)$ to $Q.$ 
\end{definition}

Finally, we recall the class of Alexander quandles, which is one of the main objects studied in this paper.

\begin{definition}\cite{Nel03, EN15}
Let $M$ be an abelian group and let $\varphi \in \mathrm{Aut}(M)$. The
\emph{Alexander quandle} associated with the pair $(M,\varphi)$ is the quandle $Q(M,\varphi)$ whose underlying set is equal to $M$ and whose binary operation $*_{\varphi}$ is given by $x*_{\varphi} y =\varphi(x)+(1-\varphi)(y)$ for all $ x,y\in M.$
\end{definition}

Equivalently, we may view $M$ as a module over the ring of Laurent polynomials $\mathbb{Z}[t,t^{-1}]$, with $t$ acting as $\varphi$.

\subsection{Abelian varieties over finite fields}
In this section, we summarize some facts about abelian varieties over finite fields, following \cite{EGMpdf}.
\begin{definition}\cite[Definition (1.3)]{EGMpdf}
    An \emph{abelian variety} over a field $k$ is a group variety, which is complete as an algebraic variety.
\end{definition}
Recall that if $A$ is an abelian variety over a field $k,$ then the set $A(k)$ of $k$-rational points of $A$ is an abelian group (see \cite[Corollary (1.15)]{EGMpdf}). Throughout this section, $q$ is a prime power i.e.\ $q=p^a$ for some prime $p$ and an integer $a \geq 1.$

\begin{definition}\cite[Definition (5.3)]{EGMpdf}
Let $A$ and $B$ be abelian varieties over a field $k.$ A homomorphism $\varphi \colon A  \rightarrow B$ is called an \emph{isogeny} if $\varphi$ is surjective and $\dim A = \dim B$. In this case, $\ker(\varphi)$ is a finite group scheme over $k.$
\end{definition}
In particular, a homomorphism $\varphi \colon A \rightarrow A$ is an isogeny if $\varphi$ is surjective. Two important examples of isogenies are the following. 
\begin{example}\label{iso ex}
    Let $A$ be an abelian variety over a field $k.$ 
    \vskip 0.1in
    (1) (\cite[Proposition (5.9)]{EGMpdf}) Let $n \ne 0$ be an integer. Then the multiplication by $n$ map $[n]_A \colon A \rightarrow A $ is an isogeny. We write $A[n]$ for $\ker ([n]_A)$.
    \vskip 0.1in
    (2) (\cite[(16.1) and (16.2)]{EGMpdf}) Assume further that $k=\mathbb{F}_q$ is a finite field. Then the $q$-power Frobenius endomorphism $\pi_A \colon A \rightarrow A$ is an isogeny.
\end{example}
Related to two isogenies given in Example \ref{iso ex}, the following facts are known.
\begin{proposition}\label{av prop}
  Let $A$ be an abelian variety of dimension $g$ over a field $k.$
    \vskip 0.1in
    (a) (\cite[Corollary (5.11)]{EGMpdf}) If $\gcd(\mathrm{char}(k), n)=1,$ then $A[n](\overline{k}) \cong \left(\Z/n\Z \right)^{2g}$ (where $\overline{k}$ is an algebraic closure of $k$).
    \vskip 0.1in
    (b) (\cite[(16.2)]{EGMpdf}) Assume further that $k=\mathbb{F}_q$ is a finite field. If $\varphi \colon A \rightarrow A$ is an endomorphism of $A,$ then $\varphi \circ \pi_A = \pi_A \circ \varphi$. In other words, $\pi_A$ commutes with all endomorphisms $\varphi$ of $A$ so that $\pi_A$ belongs to the center of the endomorphism algebra of $A.$
\end{proposition}
Furthermore, following \cite[\S16]{EGMpdf}, we will repeatedly use the standard facts that if $A$ is an abelian variety of dimension $g$ over a finite field $\F_q,$ then $P_A(t) \in \Z[t]$ is monic of degree $2g$, $P_A(0)=q^g$, $P_A(1)=\sharp A(\F_q),$ all complex zeros of $P_A$ have absolute value $\sqrt{q},$ and $P_A(t)$ satisfies the usual functional equation
$$t^{2g} \cdot P_A\left(\frac{q}{t}\right) =q^g \cdot P_A(t).$$



\section{Knot colorings and persistent colorability via Frobenius Alexander quandles}
Let $K$ be an oriented knot, and let $\mathcal{A}_K$ be the Alexander module of $K$. Let $R=\mathbb{Z}[t,t^{-1}]$ be the Laurent polynomial ring in a variable $t$. Then $\mathcal{A}_K$ is an $R$-module. For each rational prime $\ell$, let $R_{\ell} = R \otimes_{\Z} \F_{\ell} \cong \F_{\ell}[t,t^{-1}]$, and let $\mathcal{A}_{K,\ell} = \mathcal{A}_K \otimes_{R} R_{\ell} \cong \mathcal{A}_K \otimes_{\Z} \F_{\ell}.$ Then $\mathcal{A}_{K, \ell}$ is naturally an $R_{\ell}$-module. In this situation, we first recall the following fundamental notion.
\begin{definition}\cite[p.206]{Rol03}
    The \emph{Alexander polynomial} $\Delta_K (t)$ is defined to be the order of the Alexander module $\mathcal{A}_K$, which is determined up to multiplication by an element in $R$ of the form $\pm t^n$ for some $n \in \Z$.
\end{definition}
In the sequel, we choose $\Delta_{K}(t)$ to be the unique representative of the Alexander polynomial of $K$ having nonzero constant term and positive leading coefficient. In particular, with this choice, $\Delta_K(t) \in \Z[t].$ (We call this $\Delta_K(t) \in \Z[t]$ the \emph{normalized Alexander polynomial of $K$}.) We will continually use the standard facts that $\Delta_K(1)= \pm 1,$ $\deg \Delta_K \leq 2 g(K)$, where $g(K)$ denotes the genus of $K,$ and that $\Delta_K(t)$ is primitive and symmetric \cite{Rol03}. 
\vskip 0.1in
Now, let $q=p^a$ for some prime $p$ and an integer $a \geq 1$, and let $\F_q$ be a finite field with $q$-elements with an algebraic closure $\overline{\F}_q$. Let $A$ be an abelian variety of dimension $g$ over $\F_q,$ and let $\pi_A$ denote the Frobenius endomorphism of $A.$ For each prime $\ell \ne p,$ let $\pi_{A,\ell}$ be the restriction of $\pi_A$ to $A[\ell](\overline{\F}_q).$ Also, let $P_A(t)$ denote the characteristic polynomial of $\pi_A.$ 

Let $V_{K,A,\ell} = \mathrm{Hom}_{\mathrm{Quand}}(Q(K),Q(A[\ell](\overline{\F}_q),\pi_{A,\ell}))$ be the set of quandle homomorphisms from the fundamental quandle $Q(K)$ of $K$ to the Alexander quandle $Q(A[\ell](\overline{\F}_q),\pi_{A,\ell})$. Then $V_{K,A,\ell}$ is an abelian group under the binary operation $+$ defined by $(f+g)(x)=f(x)+g(x)$ for all $f, g \in V_{K,A,\ell}$ and for all $x \in Q(K)$, and in fact, $V_{K,A,\ell}$ is also an Alexander quandle together with the quandle operation $*_{\widetilde{\pi}_{A,\ell}}$ defined by $f *_{\widetilde{\pi}_{A,\ell}} g = \widetilde{\pi}_{A,\ell} (f) + (1-\widetilde{\pi}_{A,\ell})(g)$ for all $f, g \in V_{K,A,\ell}$, where $\widetilde{\pi}_{A,\ell} (f)=\pi_{A,\ell} \circ f$ for any $f \in V_{K,A,\ell}.$ Let $x \in Q(A[\ell](\overline{\F}_q),\pi_{A,\ell})$ be an arbitrary element, and let $c_x \colon Q(K) \rightarrow Q(A[\ell](\overline{\F}_q),\pi_{A,\ell})$ be the constant map given by $c_x (a) = x$ for all $a \in Q(K).$ Then $c_x \in V_{K,A,\ell},$ and then since $\sharp V_{K,A,\ell} \geq \sharp A[\ell](\overline{\F}_q)  = \ell^{2g} \geq 2,$ it follows that $V_{K,A,\ell}$ is nontrivial.
\vskip 0.1in
The following fact considers the connectedness of the Alexander quandle $V_{K,A,\ell}.$
\begin{lemma}\label{conn lem}
    For a rational prime $\ell \ne p$, $V_{K,A,\ell}$ is a connected Alexander quandle if and only if $\ell$ does not divide $\sharp A(\F_q)$.
\end{lemma}
\begin{proof}
     In light of \cite[Definition 2.3]{Nel03}, the Alexander quandle $V_{K,A,\ell}$ is connected if and only if $1- \widetilde{\pi}_{A,\ell} \colon V_{K,A,\ell} \rightarrow V_{K,A,\ell}$ is surjective. Since $V_{K,A,\ell}$ is finite, the latter is equivalent to $1-\widetilde{\pi}_{A,\ell}$ being injective. 
     \vskip 0.1in
     We first note that the map $1-\widetilde{\pi}_{A,\ell}$ is injective if and only if the map $1-\pi_{A,\ell} \colon A[\ell](\overline{\F}_q) \rightarrow A[\ell](\overline{\F}_q)$ is injective. Indeed, by definition, if $1-\pi_{A,\ell}$ is injective, then $1-\widetilde{\pi}_{A,\ell}$ is injective. For the converse, let $x \in \ker(1-\pi_{A,\ell}).$ Consider the constant map $c_x \colon Q(K) \rightarrow Q(A[\ell](\overline{\F}_q), \pi_{A,\ell})$ given by $c_x(a)=x$ for all $a \in Q(K)$, which is an element in $V_{K,A,\ell}$. Then $c_x \in \ker(1-\widetilde{\pi}_{A,\ell}),$ and since $1-\widetilde{\pi}_{A,\ell}$ is injective, it follows that $c_x$ is the zero map so that $x =0.$ Thus $1-\pi_{A,\ell}$ is injective. 
     \vskip 0.1in
     Now, we claim that $1-\pi_{A,\ell}$ is injective if and only if $\ell$ does not divide $\sharp A(\F_q)$. Indeed, since $\pi_{A,\ell}$ fixes all $\F_q$-rational points of $A,$ we have 
     $$\ker(1-\pi_{A,\ell})=\{x \in A[\ell](\overline{\F}_q) : \pi_{A,\ell}(x)=x \}=A[\ell](\F_q).$$
     It follows that $1-\pi_{A,\ell}$ is injective if and only if $A[\ell](\F_q) = 0$, and the latter is equivalent to the fact that $\ell$ does not divide $\sharp A(\F_q).$
     \vskip 0.1in
     This completes the proof.
\end{proof}
\begin{remark}\label{Q(Mell) conn lem}
As we can see in the proof of Lemma \ref{conn lem}, we also have that, for a rational prime $\ell \ne p,$ $Q(A[\ell](\overline{\F}_q),\pi_{A,\ell})$ is a connected Alexander quandle if and only if $\ell$ does not divide $\sharp A(\F_q).$
\end{remark}

Since $\sharp A(\F_q) =P_A(1)$, we obtain the following corollary.
\begin{corollary}
    If $P_A(1)=1,$ then $V_{K,A,\ell}$ is a connected Alexander quandle for any prime $\ell \ne p$ and for all oriented knots $K.$
\end{corollary}

\begin{example}\label{1rationalpt ex}
Let $A \colon y^2 +y = x^3 +x +1$ be an elliptic curve over $\F_2$. Then we have $P_A(1)=\sharp A(\F_2) =1$, and hence, the Alexander quandle $V_{K,A,\ell}$ is connected for all odd primes $\ell $ and all oriented knots $K.$    
\end{example}

Now, let $\mathcal{C}_{A,\ell}(K)=\mathrm{Hom}_{R_\ell} (\mathcal A_{K,\ell},A[\ell](\overline{\mathbb F}_q))$. Fix an oriented diagram $D$ of $K$ with arcs $a_1,\cdots,a_r$, and fix an arc $a$. Then evaluation at $a$ gives an $R_\ell$-linear map
$\mathrm{ev}_a \colon V_{K,A,\ell}\to A[\ell](\overline{\mathbb F}_q)$, which is split by the constant colorings. By the standard based Alexander presentation, $\ker(\mathrm{ev}_a)\cong \mathcal{C}_{A,\ell}(K)$, where a coloring $f$ corresponds to $\phi_f$ defined by $\phi_f(\overline{a_j} )=f(a_j)-f(a)$, with $\overline{a_j} $ denoting the class of $a_j-a$. Hence the map $\psi_a \colon V_{K,A,\ell} \rightarrow A[\ell](\overline{\F}_q) \oplus \mathcal{C}_{A,\ell}(K),$ defined by $\psi_{a}(f)=(f(a), \phi_{f})$ is an $R_\ell$-module isomorphism, with the inverse $\psi_a^{-1}(m,\phi)(a_j)=m+\phi(\overline{a_j})$. Consequently, $V_{K,A,\ell}\cong A[\ell](\overline{\mathbb F}_q)\oplus \mathcal{C}_{A,\ell}(K)$ as $R_{\ell}$-modules, non-canonically.

\begin{lemma}\label{reduction rel prime lem0}
   Let $\ell \ne p$ be a rational prime, and let $\overline{\Delta}_K (t)$ and $\overline{P}_A (t)$ be the reductions of $\Delta_K(t)$ and $P_A(t)$ modulo $\ell,$ respectively. Then $\mathcal{C}_{A,\ell}(K) \ne 0$ if and only if $\overline{\Delta}_K (t)$ and $\overline{P}_A (t)$ are not relatively prime in $R_{\ell}.$
\end{lemma}
\begin{proof}
Since $R_{\ell}$ is a PID, we may write
$$ \mathcal A_{K,\ell}
\cong
\bigoplus_{i=1}^r R_\ell/(f_i(t))
\quad\text{and}\quad
A[\ell](\overline{\mathbb F}_q)
\cong
\bigoplus_{j=1}^s R_\ell/(g_j(t))$$
for some integers $r \geq 0$ and $s \geq 1$, and $f_i(t), g_j(t) \in R_{\ell}$, by the structure theorem for finitely generated torsion $R_\ell$-modules. (Here, we use the convention that $\mathcal{A}_{K,\ell}=0$ for $r=0$ in the above direct sum.) The orders of these $R_\ell$-modules are, up to units,
$\overline{\Delta}_K(t)$ and $\overline P_A(t)$, respectively, and hence, we have
\begin{equation}\label{eqn 1}
  \overline{\Delta}_K(t)=\prod_{i=1}^r f_i(t),
\quad \textrm{and} \quad 
\overline P_A(t)=\prod_{j=1}^s g_j(t).  
\end{equation}
up to units in $R_\ell$. Hence it follows that
$$ \mathcal C_{A,\ell}(K)
=
\mathrm{Hom}_{R_\ell}
(\mathcal A_{K,\ell},A[\ell](\overline{\mathbb F}_q))
\cong
\bigoplus_{i,j}
\mathrm{Hom}_{R_\ell}
(R_\ell/(f_i(t)),R_\ell/(g_j(t))).$$
Now, since $\mathrm{Hom}_{R_{\ell}}(R_{\ell}/(f_i), R_{\ell}/(g_j)) \ne 0$ if and only if $f_i$ and $g_j$ are not relatively prime in $R_{\ell}$, for any $i,j$, it follows that $\mathcal C_{A,\ell}(K)\neq0$ if and only if $f_i$ and $g_j$ have a nonunit common factor for some $i,j$,
which is equivalent to saying, in view of (\ref{eqn 1}), that $\overline{\Delta}_K(t)$ and $\overline{P}_A(t)$ not being relatively prime in $R_{\ell}$.
\vskip 0.1in
This completes the proof.
\end{proof}

\begin{theorem}\label{Rl mod isom thm}
The following statements are equivalent.
\vskip 0.1in
(i) $P_A(t)$ and $\Delta_K(t)$ are relatively prime in $\mathbb Q[t].$
\vskip 0.1in
(ii) $\mathcal C_{A,\ell}(K)=0$ for some rational prime
$\ell\neq p$.
\vskip 0.1in
(iii) $\mathcal C_{A,\ell}(K)=0$ for all but finitely many
rational primes $\ell\neq p$.
\vskip 0.1in
\noindent Equivalently, in (ii) and (iii), one may replace
$\mathcal C_{A,\ell}(K)=0$ by
$$V_{K,A,\ell}\cong A[\ell](\overline{\mathbb F}_q)$$
as $R_\ell$-modules.
\end{theorem}


\begin{proof}
Suppose first that $P_A(t)$ and $\Delta_K(t)$ are relatively prime
in $\mathbb Q[t]$. Since $\Q[t]$ is a PID, there exist $a_0(t),b_0(t)\in\mathbb Q[t]$ such that
$$a_0(t)P_A(t)+b_0(t)\Delta_K(t)=1.$$
Then by clearing denominators, we may write
$$a(t) P_A(t) + b(t) \Delta_K(t)= N$$
for some $a(t), b(t) \in \Z[t]$ and a nonzero integer $N.$ Hence, for every prime $\ell\nmid N$, reduction modulo $\ell$ gives
\[
\overline a(t)\,\overline P_A(t)
+
\overline b(t)\,\overline\Delta_K(t)
=
\overline N\neq0.
\]
Since $\overline N$ is a unit in $\mathbb F_\ell$, it follows that
$\overline P_A(t)$ and $\overline\Delta_K(t)$ are relatively prime
in $R_\ell$. Hence, we get $\mathcal C_{A,\ell}(K)=0$ for all but finitely many primes $\ell\neq p$ by Lemma \ref{reduction rel prime lem0}. This proves (i) $\Rightarrow$ (iii).
\vskip 0.1in
Clearly, (iii) implies (ii). It remains to prove that (ii) implies (i). To this aim, suppose on the contrary that $P_A(t)$ and $\Delta_K(t)$ are not relatively prime in $\mathbb Q[t]$. Since $P_A(t)$ is monic, they have a nonconstant monic common factor $h(t)\in\mathbb Z[t]$. Write $$P_A(t)=h(t)u(t)$$
with $u(t)\in\mathbb Z[t]$. Since $P_A(0)=q^g = p^{ag}$, we have $h(0)u(0)=p^{ag}.$ Thus, for every prime $\ell\neq p$, we have $\ell\nmid h(0)$. Consequently, the reduction $\overline h(t)$ is a nonconstant nonunit of $R_\ell$. Since $\overline h(t)$ divides both
$\overline P_A(t)$ and $\overline\Delta_K(t)$, these two
polynomials are not relatively prime in $R_\ell$. By Lemma \ref{reduction rel prime lem0} again, it follows that $\mathcal C_{A,\ell}(K)\neq0$ for every prime $\ell\neq p$, which contradicts our assumption (ii).
\vskip 0.1in

Finally, the equivalent assertion concerning $V_{K,A,\ell}$ follows
from the previously established isomorphism $V_{K,A,\ell} \cong A[\ell](\overline{\mathbb F}_q) \oplus \mathcal C_{A,\ell}(K),$ as $R_{\ell}$-modules.
\vskip 0.1in
This completes the proof.
\end{proof}

\begin{corollary}\label{quand isom cor}
    $P_A(t)$ and $\Delta_K(t)$ are relatively prime in $\Q[t]$ if and only if the quandles $Q(V_{K,A,\ell}, \widetilde{\pi}_{A,\ell})$ and $Q(A[\ell](\overline{\F}_q),\pi_{A,\ell})$ are isomorphic for all but finitely many primes $\ell \ne p.$
\end{corollary}
\begin{proof}
    Suppose that $P_A(t)$ and $\Delta_K(t)$ are relatively prime in $\Q[t]$. Then by Theorem \ref{Rl mod isom thm}, there exists an $R_{\ell}$-module isomorphism $\Phi_{K,A,\ell} \colon V_{K,A,\ell} \rightarrow A[\ell](\overline{\F}_q)$ for all but finitely many primes $\ell \ne p.$ Since $t$ acts on $V_{K,A,\ell}$ (resp.\ on $A[\ell](\overline{\F}_q)$) as $\widetilde{\pi}_{A,\ell}$ (resp.\ $\pi_{A,\ell})$, it follows that we have
    $$ \pi_{A,\ell} \circ \Phi_{K,A,\ell} =  \Phi_{K,A,\ell} \circ \widetilde{\pi}_{A,\ell}, $$
    and hence, $\Phi_{K,A,\ell}$ is indeed a quandle isomorphism. Conversely, suppose that for all but finitely many primes $\ell \ne p,$ $Q(V_{K,A,\ell}, \widetilde{\pi}_{A,\ell}) \cong Q(A[\ell](\overline{\F}_q),\pi_{A,\ell})$ as quandles. In view of Theorem \ref{Rl mod isom thm}, it suffices to show that $V_{K,A,\ell} \cong A[\ell](\overline{\F}_q)$ as $R_{\ell}$-modules. Indeed, by our assumption, we have $\sharp V_{K,A,\ell}=\sharp A[\ell](\overline{\F}_q).$ Since $V_{K,A,\ell} \cong A[\ell](\overline{\F}_q) \oplus \mathrm{Hom}_{R_{\ell}}(\mathcal{A}_{K,\ell}, A[\ell](\overline{\F}_q))$ as $R_{\ell}$-modules, as established in the above, it follows that $\mathrm{Hom}_{R_{\ell}}(\mathcal{A}_{K,\ell}, A[\ell](\overline{\F}_q))=0.$ Then $V_{K,A,\ell} \cong A[\ell](\overline{\F}_q)$ as $R_{\ell}$-modules, and hence, $P_A(t)$ and $\Delta_K(t)$ are relatively prime in $\Q[t]$ again by Theorem \ref{Rl mod isom thm}.
    \vskip 0.1in
    This completes the proof.
\end{proof}
\vskip 0.1in

Now, for fixed $A$ and $K,$ let $\Sigma_A(K)=\{\ell \ne p \in \mathcal{P}~|~ \mathcal{C}_{A,\ell}(K) \ne 0 \}.$ Then the following result explains when a rational prime $\ell \ne p$ belongs to $\Sigma_{A}(K)$ in terms of the reduced Alexander polynomial of $K$ and the reduced characteristic polynomial of $A$ modulo $\ell.$ 
\vskip 0.1in

We introduce our first main result of this paper. 
\begin{theorem}\label{main thm A}
    Let $R_{A,K} = \mathrm{Res}(\Delta_K(t),P_A(t)) \in \Z$ be the resultant of $\Delta_K(t)$ and $P_A(t).$ Then exactly one of the following two cases occurs.
    $$(i)~R_{A,K}=0~\textrm{and}~\Sigma_A(K)=\{\ell \in \mathcal{P} : \ell \ne p \}.$$
    $$(ii)~R_{A,K} \ne 0~\textrm{and}~\Sigma_A(K)=\{\ell \in \mathcal{P} : \ell \ne p~\textrm{and}~\ell \mid R_{A,K} \}, \textrm{which is finite}.$$
\end{theorem}
   
\begin{proof}
By Lemma \ref{reduction rel prime lem0}, a rational prime $\ell\ne p$ belongs to $\Sigma_A(K)$
if and only if $\overline{\Delta}_K(t)$ and $\overline{P}_A(t)$ are not relatively prime in $R_{\ell}$. Since $P_A(t)$ is monic, the latter is equivalent to $\ell\mid R_{A,K}$.

    \vskip 0.1in
    Now, there are two cases to consider. If $R_{A,K}=0,$ then any rational prime $\ell \ne p$ divides $R_{A,K}=0,$ so that we have $\Sigma_{A}(K)=\{\ell \in \mathcal{P} : \ell \ne p \}.$ If $R_{A,K} \ne 0$, then we see that $\Sigma_{A}(K)=\{\ell \in \mathcal{P} : \ell \ne p, \ell \mid R_{A,K}\}$ is finite. 
\vskip 0.1in
 This completes the proof.
\end{proof}
In view of Theorem \ref{main thm A}, we introduce the following notion.
\begin{definition}
A knot $K$ is said to be \emph{persistently $A$-colorable} if $\Sigma_{A}(K)$ is infinite.
\end{definition}
Then according to Theorem \ref{main thm A}, a knot $K$ is persistently $A$-colorable if and only if $\Sigma_{A}(K) = \{\ell \in \mathcal{P} : \ell \ne p \}.$ The next result provides a criterion for exactly when a knot $K$ is persistently $A$-colorable.
\begin{theorem}\label{main thm B}
    The following four statements are equivalent.
    \vskip 0.1in
    (i) $K$ is persistently $A$-colorable.
    \vskip 0.1in
    (ii) $\Sigma_{A}(K) = \{\ell \in \mathcal{P}: \ell \ne p \}.$
    \vskip 0.1in
    (iii) $\Delta_K(t)$ and $P_A(t)$ are not relatively prime in $\Q[t].$  
    \vskip 0.1in
    (iv) There exists a simple abelian variety $B$ over $\F_q$, which is an isogeny factor of $A$ over $\F_q$ such that $\sharp B(\F_q)=1$ and $m_B(t) \mid \Delta_K(t)$ in $\Q[t]$, where $m_B(t)$ denotes the minimal polynomial of the Frobenius endomorphism $\pi_B$ of $B.$ 
\end{theorem}
\begin{proof}
    (i) $\Rightarrow$ (ii) follows from Theorem \ref{main thm A}. (ii) $\Rightarrow $ (iii) follows from Theorem \ref{Rl mod isom thm}.  
    \vskip 0.1in
    Now, let $A \sim B_1^{m_1} \times B_2^{m_2} \times \cdots \times B_r^{m_r}$ be the isogeny decomposition of $A$ over $\F_q$, where $B_1,B_2,\cdots, B_r$ are simple abelian varieties over $\F_q$, which are pairwise non-isogenous over $\F_q$, and $m_1, m_2,\cdots, m_r$ are positive integers. According to the proof of \cite[Theorem 16.4]{EGMpdf}, we have $\displaystyle P_A = \prod_{j=1}^{r} P_{B_j}^{m_j}$ where $P_{B_{j}}$ denotes the characteristic polynomial of $\pi_{B_j}$ for each $j=1,2,\cdots,r.$ Then by assumption, there exists $i \in \{1,2,\cdots,r\}$ such that the minimal polynomial $m_{B_i}(t)$ of $\pi_{B_i}$ divides $\Delta_K (t)$ in $\Q[t].$ Since $m_{B_i}(t)$ is monic and $\Delta_K (t)$ is primitive, it follows from Gauss's lemma that $m_{B_i}(t) \mid \Delta_K(t)$ in $\Z[t].$ In particular, $m_{B_i}(1)  \mid \Delta_K(1)=\pm 1$, and hence, we have $m_{B_i}(1) = \pm 1.$ Since $P_{B_i}(t)$ is a power of $m_{B_i}(t)$ and $P_{B_i}(1)=\sharp B_i(\F_q) \geq 1$, it follows that $P_{B_i}(1)=\sharp B_i(\F_q)=1.$ This proves (iii) $\Rightarrow$ (iv). 
    \vskip 0.1in
    Finally, by assumption, the irreducible polynomial $m_B(t)$ divides both $P_A(t)$ and $\Delta_K(t)$, and hence, $K$ is persistently $A$-colorable by Theorem \ref{main thm A}. This proves (iv) $\Rightarrow$ (i).
    \vskip 0.1in
    This completes the proof.
\end{proof}

\begin{corollary}\label{main cor A}
    If $|\alpha| \ne \sqrt{q}$ for all zeros $\alpha$ of $\Delta_K(t),$ then $V_{K,A,\ell} \cong A[\ell](\overline{\F}_q)$ for all but finitely many primes $\ell \ne p.$
\end{corollary}
\begin{proof}
    By assumption, we have $\gcd(\Delta_{K}(t), P_A(t))=1,$ and hence, it follows from Theorems \ref{main thm A} and \ref{main thm B} that $\Sigma_A(K)= \{\ell \ne p \in \mathcal{P} : \ell \mid R_{A,K} \}.$ This implies that we have $\mathcal{C}_{A,\ell}(K) \ne 0$ only for those finite rational primes $\ell \ne p$ dividing $R_{A,K},$ which completes the proof.
\end{proof}
\begin{example}
   Let $K=3_1$ be the trefoil knot. Then it is known that $\Delta_K(t)=t^2 - t+1$. Since the zeros of $\Delta_K(t)=t^2 -t +1$ are primitive sixth roots of unity, it follows from Corollary \ref{main cor A} that $V_{K,A,\ell} \cong A[\ell](\overline{\F}_q)$ for all but finitely many primes $\ell \ne p.$
\end{example}

\begin{corollary}\label{main cor B}
    If $K$ is persistently $A$-colorable, then $q \leq 4.$ In particular, $p=2$ or $p=3.$
\end{corollary}
    
\begin{proof}
By Theorem \ref{main thm B}, there exists a simple isogeny factor $B$ of $A$ such that $\sharp B(\F_q)=1$ and $m_B(t) \mid \Delta_K(t)$. In particular, $m_B(1) \mid \Delta_K(1)= \pm 1$, and hence, we have $m_B(1)=\pm 1.$ Then since $P_B(t)$ is a power of $m_B(t)$ and $P_B(1)=\sharp B(\F_q) \geq 1,$ it follows that $P_B(1)=\sharp B(\F_q)=1.$ By Hasse-Weil bound, we have
$$(\sqrt{q}-1)^{2 \dim B} \leq 1,$$
and hence, we obtain that $q \leq 4,$ as desired. Also, since $q$ is a power of $p,$ $p=2$ or $p=3.$ 
\vskip 0.1in
This completes the proof.
\end{proof}
Corollary \ref{main cor B} indicates that if the order of the finite base field of $A$ is large, then $K$ is not persistently $A$-colorable as in the following corollary.
\begin{corollary}\label{main cor C}
     If $q \geq 5,$ then for any abelian variety $A$ over $\F_q$ and any oriented knot $K,$ we have $V_{K,A,\ell} \cong A[\ell](\overline{\F}_q)$ as $R_{\ell}$-modules for all but finitely many primes $\ell \ne p$. (In other words, for all but finitely many rational primes $\ell \ne p,$ the knot $K$ admits only constant colorings by the Alexander quandle $Q(A[\ell](\overline{\F}_q),\pi_{A,\ell}).$)
\end{corollary}
\begin{proof}
    By Corollary \ref{main cor B} and Theorem \ref{main thm B}, we have $\gcd(\Delta_K(t), P_A(t)) =1,$ and hence, the desired result follows from Theorem \ref{Rl mod isom thm}.
\end{proof}

The bound $q \geq 5$ in Corollary \ref{main cor C} is sharp, as shown in the next proposition.
\begin{proposition}\label{main prop A}
    For each $q \in \{2,3,4\},$ there exists an abelian variety $A$ over $\F_q$ and an oriented knot $K$, such that $K$ is persistently $A$-colorable. 
\end{proposition}
\begin{proof}
    For $q=2$, in view of \cite[Theorem 16.64]{EGMpdf}, let $A$ be an elliptic curve over $\F_q$ with $\sharp A(\F_q)=1$ and $m_A(t)=t^2 -2t+2$. Also, let $f(t)=2t^4 -6t^3 +9t^2 -6t+2.$ Since $f(t)$ is symmetric up to multiplication by a power of $t$, and satisfies $f(1)=1,$ by the classical realization theorem for Alexander polynomials \cite[Theorem 5 in p. 171]{Rol03}, there exists an oriented knot $K$ with $\Delta_K(t)=f(t) = 2t^4 -6t^3 +9t^2 -6t+2.$ 
    Then since $m_A(t)$ divides $\Delta_K(t),$ it follows from Theorem \ref{main thm B} that $K$ is persistently $A$-colorable. 
    \vskip 0.1in
    Similarly, for $q=3$ (resp.\ $q=4$), we take $A$ to be an elliptic curve over $\F_q$ with $\sharp A(\F_q)=1$ and $m_A(t) = t^2 -3t+3$ (resp.\ $m_A(t)=t-2$), and $f(t)=3t^4 -12t^3 +19t^2 -12t +3 $ (resp.\ $f(t)=2t^2 - 5t+2$). Then we can proceed as in the case of $q=2.$
    \vskip 0.1in
    This completes the proof.
\end{proof}

In fact, the statements in Proposition \ref{main prop A} can be improved as in the following remark.
\begin{remark}\label{main rmk 2}
    For each $q \in \{2,3,4\},$ there exists an oriented knot $K$ such that, for every $g \geq 1,$ there exists an abelian variety $A$ of dimension $g$ over $\F_q$ for which $K$ is persistently $A$-colorable. Moreover, we may choose $K$ independently of $q.$ 
\end{remark}

Now, we apply our previous results to various families of knots. We begin our discussion with knots whose normalized Alexander polynomials are monic.
\begin{theorem}\label{main thm C}
    Let $K$ be an oriented knot whose normalized Alexander polynomial $\Delta_{K}(t) \in \Z[t]$ is monic. Then $K$ is not persistently $A$-colorable for any abelian variety $A$ over any finite field $\F_q.$
\end{theorem}
\begin{proof}
    Suppose on the contrary that there exists an abelian variety $A$ over a finite field $\F_q$ such that $K$ is persistently $A$-colorable. Then by Theorem \ref{main thm B}, there exists an $\F_q$-simple isogeny factor $B$ of $A$ such that $\sharp B(\F_q)=1$ and $m_B(t) \mid \Delta_K(t)$ in $\Q[t].$ Then since $m_B(t)$ is monic and $\Delta_K(t)$ is primitive, it follows from Gauss's lemma that $m_B(t) \mid \Delta_K(t)$ in $\Z[t].$ Now, since $\Delta_K(t)$ is monic and symmetric, we have $\Delta_K(0)=1,$ while $|m_B(0)|=q^{\frac{\deg m_B}{2}}>1$, which is absurd. 
    \vskip 0.1in
    This completes the proof.
\end{proof}
\begin{corollary}\label{main cor D}
    If $K$ is an oriented fibered knot, then $K$ is not persistently $A$-colorable for any abelian variety $A$ over any finite field $\F_q.$
\end{corollary}
\begin{proof}
    If $K$ is an oriented fibered knot, then it is known that the normalized Alexander polynomial $\Delta_K(t)$ is monic \cite{Rol03}, and hence, the desired result follows from Theorem \ref{main thm C}.
\end{proof}
\begin{example}
    Every oriented torus knot $T_{r,s}$ (where $r, s>1$ are coprime integers) is not persistently $A$-colorable for any abelian variety $A$ over any finite field $\F_q.$
\end{example}

\begin{theorem}\label{main thm D}
 Let $A$ be an abelian variety over a finite field $\F_q$ and let $K$ be an oriented knot of genus $g(K)$. Then $K$ is persistently $A$-colorable if and only if $q \in \{2,3,4 \}$, $A$ has an $\F_q$-simple isogeny factor $B$ with $\sharp B(\F_q)=1,$ $\deg m_B \leq g(K),$ and $m_B(t) \cdot m_B^{\vee}(t) \mid \Delta_K(t)$ in $\Q[t]$, where $m_B^{\vee}(t)=\frac{t^{\deg m_B} \cdot  m_B(t^{-1})}{m_B(0)} \in \Q[t].$    
\end{theorem}

\begin{proof}
Suppose first that $K$ is persistently $A$-colorable. By Corollary \ref{main cor C},
$q\in\{2,3,4\}$, and by Theorem \ref{main thm B}, $A$ has an $\mathbb F_q$-simple
isogeny factor $B$ such that $\sharp B(\mathbb F_q)=1$ and
$m_B(t)\mid \Delta_K(t)$ in $\mathbb Q[t]$. Since $\Delta_K(t)$ is
symmetric, we also have $m_B^\vee(t)\mid\Delta_K(t)$. Moreover,
$m_B(t)\neq m_B^\vee(t)$, since otherwise
$m_B(0)^2=1$, contradicting $|m_B(0)|=q^{\frac{\deg m_B}{2}}>1$. Since $m_B(t)$ is irreducible,
we have $\gcd(m_B(t),m_B^\vee(t))=1$, and hence, $m_B(t) m_B^{\vee}(t) \mid \Delta_{K}(t)$ in $\Q[t].$ It follows that 
$$ \deg (m_B \cdot  m_B^{\vee}) = 2 \deg m_B  \leq \deg \Delta_K   \leq 2 g(K),$$
and hence, $\deg m_B \leq g(K).$
\vskip 0.1in
Conversely, the stated conditions imply $m_B(t)\mid\Delta_K(t)$ in $\Q[t]$, and
hence, $K$ is persistently $A$-colorable by Theorem \ref{main thm B}.
\vskip 0.1in
This completes the proof.
\end{proof}    

As a consequence of Theorem \ref{main thm D}, we obtain a complete characterization for oriented genus-one knots as in the following corollary.
\begin{corollary}\label{main cor E}
 Let $A$ be an abelian variety over a finite field $\F_q$ and let $K$ be an oriented genus-one knot. Then $K$ is persistently $A$-colorable if and only if $q=4$, $A$ has an $\F_q$-simple isogeny factor $B$ with $\sharp B(\F_q)=1,$ $m_B(t)=t-2,$ and $\Delta_K(t)=2t^2-5t+2.$   
\end{corollary}
\begin{proof}
  Suppose first that $K$ is persistently $A$-colorable. Then by Theorem \ref{main thm D}, $q \leq 4,$ $A$ has an $\F_q$-simple isogeny factor $B$ with $\sharp B(\F_q) =1,$ $\deg m_B = 1$, and $m_B(t) \cdot m_B^{\vee}(t) \mid \Delta_K(t)$ in $\Q[t].$ Since $m_B(t)$ is monic, write $m_B(t)=t-a$ for some $a \in \Z.$ Since $a$ is a $q$-Weil number, we have $|a| = \sqrt{q}$ so that $q=4,$ and $a = \pm 2.$ Then since
  $$ 1= \sharp B(\F_q)=P_B(1) = m_B(1)^e$$
  for some integer $e \geq 1$, we have $m_B(1)=1-a = \pm 1.$ By combining this observation with $a=\pm 2,$ we see that $a=2,$ whence $m_B(t)=t-2.$ By direct computation, we also get $m_B^{\vee}(t)=t-\frac{1}{2}$ so that $m_B(t) m_B^{\vee}(t) =t^2 - \frac{5}{2}t +1 $ divides $\Delta_K(t).$ Then since $\deg \Delta_K \leq 2g(K)=2,$ it follows that $\deg \Delta_K = 2,$ and then since $2t^2 -5t+2 \in \Z[t]$ is primitive, we get $\Delta_K(t) = 2t^2-5t+2$, as desired.  
  \vskip 0.1in
  The converse follows directly from Theorem \ref{main thm D}, and this completes the proof.
\end{proof}
\begin{corollary}\label{main cor F}
Let $A$ be an abelian variety over a finite field $\F_q$ and let $(J_n)_{n \geq 2}$ be the family of oriented twist knots with $2n$ half-twists, so that $\Delta_{J_n}(t) = n t^2 -(2n+1)t +n \in \Z[t]$. Then $J_n$ is persistently $A$-colorable if and only if $n=2, q=4$, and $A$ has an $\F_q$-simple isogeny factor $B$ with $\sharp B(\F_q)=1.$ In particular, among twist knots $J_n$ with $n\geq 2$, the stevedore knot $J_2 = 6_1$ is the unique knot that is persistently $A$-colorable for some abelian variety $A$ over a finite field $\F_q$.
\end{corollary}
\begin{proof}
Since $g(J_n)=1$ for all $n \geq 1,$ this follows from Corollary \ref{main cor E}.    
\end{proof}

For the case of oriented genus-two knots, we have the following result.
\begin{corollary}\label{main cor G}
Let $A$ be an abelian variety over a finite field $\F_q$ and let $K$ be an oriented genus-two knot. Then $K$ is persistently $A$-colorable if and only if one of the following holds:
\vskip 0.1in
(i) $q=2$, $A$ has a simple isogeny factor $B$ with $\sharp B(\F_2)=1$, $m_B(t)=t^2-2t+2,$ and $\Delta_K(t)=2t^4 - 6t^3 +9t^2 -6t + 2.$
\vskip 0.1in
(ii) $q=2$, $A$ has a simple isogeny factor $B$ with $\sharp B(\F_2)=1$, $m_B(t)=t^2 -2,$ and $\Delta_K(t)=2t^4 -5t^2+ 2.$
\vskip 0.1in
(iii) $q=3$, $A$ has a simple isogeny factor $B$ with $\sharp B(\F_3)=1$, $m_B(t)=t^2-3t+3,$ and $\Delta_K(t)=3t^4 - 12t^3 +19t^2 -12t + 3.$
\vskip 0.1in
(iv) $q=4$, $A$ has a simple isogeny factor $B$ with $\sharp B(\F_4)=1$, $m_B(t)=t-2,$ and $2t^2 -5t +2 \mid \Delta_K(t).$
\end{corollary}
\begin{proof}
    By Theorem \ref{main thm D}, we only need to consider $\F_q$-simple factors $B$ of $A$ with $\sharp B(\F_q)=1$ and $\deg m_B \leq 2.$ By a similar argument as in the proof of Corollary \ref{main cor E}, the possibilities for $m_B(t)$ are precisely given by
    $$t^2 - 2t+2, ~~~t^2 -2,~~~t^2-3t+3,~~~\textrm{and}~~~t-2$$
    for $q=2,2,3,$ and $4,$ respectively. Again by direct computation, the possibilities for $m_B(t) m_B^{\vee}(t) $ are precisely given by
    $$t^4 -3t^3 +\frac{9}{2}t^2 -3t +1,~~~t^4 -\frac{5}{2}t^2 +1,~~~t^4 -4t^3 + \frac{19}{3}t^2 -4t +1,~~~\textrm{and}~~~t^2 -\frac{5}{2}t +1$$
    for $q=2,2,3,$ and $4,$ correspondingly. In the first three cases, since $\deg \Delta_K \leq 4$ and $\Delta_K(t)$ is primitive, we get the respective desired $\Delta_K(t)$ in the corollary. In the last case, we just obtain the divisibility $2t^2 -5t +2 \mid \Delta_K(t).$ 
    \vskip 0.1in
    The converse follows directly from Theorem \ref{main thm D}, and this completes the proof.
\end{proof}


Our next result is on the behavior of persistent $A$-colorability under connected sums of two knots.
\begin{proposition}\label{main prop B}
Let $A$ be an abelian variety over a finite field $\F_q$, and let $K_1$ and $K_2$ be two oriented knots. Then the connected sum $K_1 \# K_2$ is persistently $A$-colorable if and only if at least one of $K_1$ and $K_2$ is persistently $A$-colorable.    
\end{proposition}
\begin{proof}
    Recall that we have $\Delta_{K_1 \# K_2} (t) = \Delta_{K_1}(t) \cdot \Delta_{K_2}(t).$ Now, suppose first that the connected sum $K_1 \# K_2$ is persistently $A$-colorable. Then by Theorem \ref{main thm B}, there is a simple isogeny factor $B$ of $A$ such that $\sharp B(\F_q)=1$ and $m_B(t) \mid \Delta_{K_1 \# K_2}(t)=\Delta_{K_1}(t) \cdot \Delta_{K_2}(t).$ Since $m_B(t)$ is irreducible over $\Q$, it follows that either $m_B(t) \mid \Delta_{K_1}(t)$ or $m_B(t) \mid \Delta_{K_2}(t).$ Then again by Theorem \ref{main thm B}, at least one of $K_1$ and $K_2$ is persistently $A$-colorable. The converse is similar.
    \vskip 0.1in
    This completes the proof.
\end{proof}

\begin{corollary}
    For each $q \in \{2,3,4 \},$ there exist an abelian variety $A$ over $\F_q$ and infinitely many oriented knots that are persistently $A$-colorable.
\end{corollary}
\begin{proof}
    By Proposition \ref{main prop A}, there exists an abelian variety $A$ over $\F_q$ and an oriented knot $K$ such that $K$ is persistently $A$-colorable. Let $(L_n)_{n \geq 1}$ be a family of pairwise distinct oriented prime knots.  Then by Schubert's prime decomposition theorem for knots \cite{Schubert} and Proposition \ref{main prop B}, the family $(K \# L_n)_{n \geq 1}$ gives rise to infinitely many distinct oriented knots which are persistently $A$-colorable. 
\end{proof}



\section{Frobenius Alexander quandles and isogeny criteria}
In this section, we establish rigidity and reconstruction results for Frobenius Alexander quandles and derive criteria for determining whether two abelian varieties over a finite field are isogenous. Throughout this section, let $A$ and $B$ be abelian varieties over a finite field $\F_q$ with $p=\mathrm{char}(\F_q)$, and let $\ell \ne p$ be a rational prime, unless otherwise stated. In the previous sections, the Frobenius Alexander quandle $Q(A[\ell](\overline{\F}_q), \pi_{A,\ell})$ was used as the target of knot colorings. In this section, we study some arithmetic properties of this target quandle itself. To this aim, we first recall the following classical result of Nelson.
\begin{proposition}[\cite{Nel03}]\label{main prop C}
    Two finite Alexander quandles $M$ and $N$ of the same cardinality are isomorphic as quandles if and only if there exists an isomorphism of $R = \Z[t, t^{-1}]$-modules $\varphi \colon (1-t)M \rightarrow (1-t)N.$
\end{proposition}

Using Proposition \ref{main prop C}, we obtain the following rigidity theorem.
\begin{theorem}\label{main thm F}
    Let $\ell \ne p$ be a rational prime. Let $V$ and $W$ be finite dimensional vector spaces over $\F_{\ell},$ and let $S \in \mathrm{GL}(V)$ and $T \in \mathrm{GL}(W)$. Then the following statements are equivalent.
    \vskip 0.1in
    (i) $Q(V,S) \cong Q(W,T)$ as quandles.
    \vskip 0.1in
    (ii) $V \cong W$ as $R_{\ell}=\F_{\ell}[t,t^{-1}]$-modules, where $t$ acts on $V$ (resp.\ $W$) via $S$ (resp.\ $T$).
    \vskip 0.1in
    (iii) There exists an isomorphism $\varphi \colon V \rightarrow W$ of $\F_{\ell}$-vector spaces such that $T \circ \varphi = \varphi \circ S.$
\end{theorem}
\begin{proof}
    The equivalence of (ii) and (iii) follows from $\F_{\ell}[t,t^{-1}]$-module structure of $V$ and $W.$ Now, suppose first that there exists an $R_{\ell}$-module isomorphism $\varphi \colon V \rightarrow W$. Then again by the $R_{\ell}$-module structure of $V$ and $W,$ we get $T \circ \varphi = \varphi \circ S.$ It follows that we have
    $$\varphi (x *_S y) = \varphi (S(x)+(1-S)(y)) = \varphi(S(x))+\varphi((1-S)(y)) = T(\varphi(x))+(1-T)(\varphi(y)) = \varphi(x)*_T \varphi(y),$$
    and hence, $\varphi$ induces an isomorphism $Q(V,S) \cong Q(W,T)$ as quandles. This proves (ii) $\Rightarrow$ (i). Finally, it remains to show that (i) implies (ii). To this aim, suppose that $Q(V,S) \cong Q(W,T)$ as quandles. Then clearly, $\sharp V = \sharp W,$ and hence, we get $\dim_{\F_{\ell}} V = \dim_{\F_{\ell}} W.$ By Proposition \ref{main prop C}, there exists a $\Lambda=\Z[t,t^{-1}]$-module isomorphism $\varphi \colon (1-t)V \rightarrow (1-t)W.$ Since both $\Lambda$-modules are annihilated by $\ell$, and since $R_{\ell} = \Lambda/\ell \Lambda,$ we can see that $\varphi$ is also an $R_{\ell}$-module isomorphism. Now, since $R_{\ell}$ is a PID, and $V,W$ are finitely generated $R_{\ell}$-modules, by the fundamental theorem on finitely generated modules over PID, we may write
    $$V \cong V_{t-1} \oplus V^{\prime}~~~\textrm{and}~~~W \cong W_{t-1} \oplus W^{\prime} $$
    where $V_{t-1} = \{v \in V : (t-1)^n  v = 0~\textrm{for some}~n \geq 1 \} \cong \bigoplus_{j=1}^{\infty} (R_{\ell}/(t-1)^j )^{m_j}$ and $W_{t-1} = \{w \in W : (t-1)^n w = 0~\textrm{for some}~n \geq 1 \}\cong \bigoplus_{j=1}^{\infty} (R_{\ell}/(t-1)^j)^{n_j}$ for some integers $m_j, n_j.$ Then note that multiplication by $1-t$ is an automorphism on $V^{\prime}$ and $W^{\prime},$ and we also have
    $$(1-t)V_{t-1} \cong \bigoplus_{j=2}^{\infty} (R_{\ell}/(t-1)^{j-1})^{m_j}~~~\textrm{and}~~~(1-t)W_{t-1} \cong \bigoplus_{j=2}^{\infty} (R_{\ell}/(t-1)^{j-1})^{n_j}.$$
    Then since $(1-t)V \cong (1-t)W$ as $R_{\ell}$-modules, by the uniqueness of the primary decomposition, it follows that $V^{\prime} \cong W^{\prime}$ as $R_{\ell}$-modules, and $m_j = n_j$ for all $j \geq 2.$ 
    \vskip 0.1in
    Finally, since $\dim_{\F_{\ell}}V = \dim_{F_{\ell}} W ,$ we also have $m_1 = n_1.$ Hence, we see that $V_{t-1} \cong W_{t-1}$, whence $V \cong W$ as $R_{\ell}$-modules. This proves (i) $\Rightarrow$ (ii). 
    \vskip 0.1in
    This completes the proof.
\end{proof}
By taking $V=A[\ell](\overline{\F}_q), S=\pi_{A,\ell}$ and $W=B[\ell](\overline{\F}_q), T=\pi_{B,\ell}$ in Theorem \ref{main thm F}, we have the following.
\begin{corollary}\label{main cor H}
    Let $A$ and $B$ be abelian varieties over a finite field $\F_q$, and let $\ell \ne p$ be a rational prime. Then $Q(A[\ell](\overline{\F}_q),\pi_{A,\ell}) \cong Q(B[\ell](\overline{\F}_q),\pi_{B,\ell})$ as quandles if and only if there exists an isomorphism $\varphi \colon A[\ell](\overline{\F}_q) \rightarrow B[\ell](\overline{\F}_q)$ as $\F_\ell$-vector spaces such that $\pi_{B,\ell} \circ \varphi = \varphi \circ \pi_{A,\ell}.$ In particular, if $Q(A[\ell](\overline{\F}_q),\pi_{A,\ell}) \cong Q(B[\ell](\overline{\F}_q),\pi_{B,\ell})$ as quandles, then $\overline{P}_A(t) = \overline{P}_B(t)$ in $\F_{\ell}[t].$
\end{corollary}
In other words, the isomorphism class of the abstract Frobenius Alexander quandle $Q(A[\ell](\overline{\F}_q),\pi_{A,\ell})$ determines the similarity class of $\pi_{A,\ell}.$ 

\begin{theorem}\label{main thm G}
    Let $A$ and $B$ be abelian varieties of dimension $g$ over $\F_q,$ and let $\Sigma$ be a finite set of rational primes different from $p.$ Assume that 
    $$\prod_{\ell \in \Sigma} \ell > 2 \binom{2g}{g} q^{\frac{g}{2}}.$$
    If $Q(A[\ell](\overline{\F}_q),\pi_{A,\ell}) \cong Q(B[\ell](\overline{\F}_q),\pi_{B,\ell})$ as quandles for all $\ell \in \Sigma,$ then $A$ and $B$ are $\F_q$-isogenous.
\end{theorem}
\begin{proof}
   Suppose that $P_A(t)=t^{2g}+a_1 t^{2g-1} + \cdots +a_{2g-1}t +  q^g$ and $P_B(t)=t^{2g}+b_1 t^{2g-1} + \cdots +b_{2g-1}t+ q^g$ for some integers $a_i, b_j$ with $i,j \in \{1,2,\cdots,2g-1\}$. By Corollary \ref{main cor H}, the mod $\ell$-reductions $\overline{P}_A(t)$ and $\overline{P}_B(t)$ are equal so that we have
   $$t^{2g} + \overline{a_1} t^{2g-1} + \cdots + \overline{q^g} = t^{2g} + \overline{b_1} t^{2g-1} + \cdots + \overline{q^g} $$
   in $\F_{\ell}[t].$ Hence $a_i \equiv b_i \pmod{\ell}$ for all $i \in \{1,2,\cdots,2g-1\}.$ Since $\ell \in \Sigma$ was arbitrary, it follows that 
   $$a_i \equiv b_i \pmod{\ell}$$
   for all $i \in \{1,2,\cdots,2g-1\}$ and all $\ell \in \Sigma$. In particular, $\prod_{\ell \in \Sigma} \ell$ divides $a_i - b_i$ for all $i \in \{1,2,\cdots, 2g-1\}.$ On the other hand, for each $1 \leq i \leq g,$ by Hasse-Weil bound, we have
   $$|a_i - b_i | \leq |a_i|+|b_i| \leq 2 \cdot \binom{2g}{i} q^{\frac{i}{2}} \leq 2 \cdot \binom{2g}{g} q^{\frac{g}{2}}< \prod_{\ell \in \Sigma} \ell.$$
   Since $\prod_{\ell \in \Sigma} \ell $ divides $a_i -b_i$ for each $1 \leq i \leq g,$ it follows that $a_i = b_i$ for all $i \in \{1,2,\cdots, g\}.$ Now, since we have
   $$a_{2g-i} = q^{g-i}a_i~~~\textrm{and}~~b_{2g-i} =q^{g-i}b_i$$
   for each $1 \leq i \leq g-1,$ we also have $a_i = b_i$ for all $i  \in \{g+1, g+2, \cdots, 2g-1 \}.$ Therefore we conclude that $P_A(t)=P_B(t)$ in $\Z[t]$, and hence, $A$ and $B$ are isogenous over $\F_q$ by \cite[Theorem 1-(c)]{Tate66}.
\vskip 0.1in
This completes the proof.
\end{proof}
   

By taking $\Sigma$ to be a singleton in Theorem \ref{main thm G}, we immediately get the following corollary.
\begin{corollary}\label{main cor I}
     Let $A$ and $B$ be abelian varieties of dimension $g$ over a finite field $\F_q$. Let $\ell \ne p$ be a rational prime with $\ell > 2 \binom{2g}{g} q^{\frac{g}{2}}.$ If $Q(A[\ell](\overline{\F}_q), \pi_{A,\ell}) \cong Q(B[\ell](\overline{\F}_q), \pi_{B,\ell})$ as quandles, then $A$ and $B$ are $\F_q$-isogenous.   
\end{corollary}
In other words, for any prime $\ell \ne p$ with $\ell >2 \binom{2g}{g} q^{\frac{g}{2}},$ the isomorphism class of the Frobenius Alexander quandle $Q(A[\ell](\overline{\F}_q),\pi_{A,\ell})$ determines the $\F_q$-isogeny class of $A$ among $g$-dimensional abelian varieties over $\F_q.$
    
\vskip 0.1in
The next corollary provides a criterion for determining whether two abelian varieties over $\F_q$ are $\F_q$-isogenous, using Frobenius Alexander quandles.
\begin{corollary}\label{main cor J}
Let $A$ and $B$ be abelian varieties over a finite field $\F_q$. Then the following are equivalent.
\vskip 0.05in
(i) $A$ and $B$ are $\F_q$-isogenous.
\vskip 0.05in

(ii) $Q(A[\ell](\overline{\F}_q), \pi_{A,\ell}) \cong Q(B[\ell](\overline{\F}_q) , \pi_{B,\ell} ) $ as quandles for infinitely many primes $\ell \ne p.$
\vskip 0.05in
(iii) $Q(A[\ell](\overline{\F}_q), \pi_{A,\ell}) \cong Q(B[\ell](\overline{\F}_q) , \pi_{B,\ell} ) $ as quandles for all but finitely many primes $\ell \ne p.$    
\end{corollary}
\begin{proof}
 We first take $\Sigma$ to be any finite subset of the set of infinitely many rational primes different from $p$ for which the two Frobenius Alexander quandles are isomorphic. Note that for every $\ell \in \Sigma,$ we have 
 $$\ell^{2 \dim A} = \sharp A[\ell](\overline{\F}_q) = \sharp B[\ell](\overline{\F}_q) = \ell^{2 \dim B}, $$
 and hence, we get $\dim A = \dim B.$ Now, by adding more rational primes if necessary, we may choose $\Sigma$ to satisfy $\prod_{\ell \in \Sigma} \ell > 2 \binom{2g}{g} q^{\frac{g}{2}}$, where $g=\dim A= \dim B.$ Then by Theorem \ref{main thm G}, $A$ and $B$ are $\F_q$-isogenous. This proves (ii) $\Rightarrow$ (i). Clearly, (iii) implies (ii). It remains to prove that (i) implies (iii). Since $A$ and $B$ are $\F_q$-isogenous, there exists an isogeny $\varphi \colon A \rightarrow B$ over $\F_q$. For every rational prime $\ell \ne p$ with $\ell \nmid \deg \varphi,$ the restriction map $\varphi_{\ell} \colon A[\ell](\overline{\F}_q) \rightarrow B[\ell](\overline{\F}_q)$ commutes with the respective Frobenius maps $\pi_{A,\ell}$ and $\pi_{B,\ell},$ and hence, it follows that $Q(A[\ell](\overline{\F}_q), \pi_{A,\ell}) \cong Q(B[\ell](\overline{\F}_q), \pi_{B,\ell})$ as quandles. Since there are only finitely many rational primes $\ell \ne p$ dividing $\deg \varphi,$ we obtain (iii). 
\vskip 0.1in
    This completes the proof.   
\end{proof}
\begin{remark}
The assumption that $A$ and $B$ are defined over the same finite field can be omitted in the following sense: let $A$ be an abelian variety over a finite field $\F_q$, and let $B$ be an abelian variety over a finite field $\F_r$. If $Q(A[\ell](\overline{\F}_q),\pi_{A,\ell}) \cong Q(B[\ell](\overline{\F}_r), \pi_{B,\ell})$ for infinitely many rational primes $\ell$ with $\ell \nmid q r,$ then $\dim A = \dim B$, $q=r,$ and $A$ and $B$ are isogenous over $\F_q.$ Indeed, the quandle isomorphisms imply that $\dim A = \dim B$, and then by Corollary \ref{main cor H}, we have
$$q^{\dim A} = P_A(0) \equiv P_B(0) = r^{\dim B} \pmod{\ell}$$
for infinitely many rational primes $\ell$, and hence, it follows that $r=q.$ Finally, $A$ and $B$ are $\F_q$-isogenous by Corollary \ref{main cor J}.
\end{remark}


 

    

\begin{corollary}\label{main cor K}
    Let $A$ and $B$ be abelian varieties of dimension $g$ over a finite field $\F_q$, and let $K$ be an oriented knot. Assume that $R_{A,K} \ne 0$ and $R_{B,K} \ne 0.$ Let $\Sigma$ be a finite set of rational primes $\ell \ne p$ such that $\ell \nmid R_{A,K}\cdot R_{B,K}$ and $\prod_{\ell \in \Sigma} \ell > 2 \binom{2g}{g} q^{\frac{g}{2}}.$ If $Q(V_{K,A,\ell}, \widetilde{\pi}_{A,\ell}) \cong Q(V_{K,B,\ell}, \widetilde{\pi}_{B,\ell})$ as quandles for all $\ell \in \Sigma,$ then $A$ and $B$ are isogenous over $\F_q.$ 
\end{corollary}
\begin{proof}
    By Theorem \ref{main thm A}, we know that $\Sigma_A(K)=\{\ell \in \mathcal{P} : \ell \ne p~\textrm{and}~\ell \mid R_{A,K} \}$ and $\Sigma_B(K)=\{\ell \in \mathcal{P} : \ell \ne p~\textrm{and}~\ell \mid R_{B,K} \}$. Then $\Sigma \cap (\Sigma_A(K) \cup \Sigma_B(K))$ is empty, and hence, for every $\ell \in \Sigma,$ we have $\mathcal{C}_{A,\ell}(K)=\mathcal{C}_{B,\ell}(K)=0.$ Hence $V_{K,A,\ell} \cong A[\ell](\overline{\F}_q)$ and $V_{K,B,\ell} \cong B[\ell](\overline{\F}_q)$ as $R_{\ell}$-modules. In particular, there exist quandle isomorphisms $Q(V_{K,A,\ell}, \widetilde{\pi}_{A,\ell}) \cong Q(A[\ell](\overline{F}_q), \pi_{A,\ell})$ and $Q(V_{K,B,\ell}, \widetilde{\pi}_{B,\ell}) \cong Q(B[\ell](\overline{F}_q), \pi_{B,\ell}).$ Now, by assumption, it follows that $Q(A[\ell](\overline{\F}_q), \pi_{A,\ell}) \cong Q(B[\ell](\overline{\F}_q), \pi_{B,\ell})$ as quandles for all $\ell \in \Sigma$, and hence, $A$ and $B$ are isogenous over $\F_q$ by Theorem \ref{main thm G}.
    \vskip 0.1in 
    This completes the proof.
\end{proof}

The next example shows that the largeness of a prime $\ell \ne p$ is essential in Corollary \ref{main cor I}.
\begin{example}
    Let $A : y^2 = x^3 + 3x+2$ and $B : y^2 = x^3 +4x$ be elliptic curves over $\F_5.$ Then by a direct computation, we can see that $\sharp A(\F_5)=5$ and $\sharp B(\F_5)=8$. Take $\ell =3 \ne 5.$ Now, the characteristic polynomials of Frobenius maps are computed as $P_A (t)=t^2 -t+5 \in \Z[t]$ and $P_B(t)=t^2 +2t+5 \in \Z[t].$ Hence $A$ and $B$ are not isogenous over $\F_5$ by \cite[Theorem 1]{Tate66}. On the other hand, modulo $\ell = 3,$ we have $\overline{P}_A(t)=\overline{P}_B(t) = t^2 + 2t +2 \in \F_{3}[t]$, which is irreducible over $\F_3,$ and hence, $A[3](\overline{\F}_5) \cong B[3](\overline{\F}_5)$ as $\F_{3}[t,t^{-1}]$-modules where $t$ acts as respective Frobenius maps. It follows that $Q(A[3](\overline{\F}_5), \pi_{A,3}) \cong  Q(B[3](\overline{\F}_5), \pi_{B,3})$ as quandles.  
\end{example}
The following example illustrates an application of Corollary \ref{main cor I}.
\begin{example}
    According to \cite{Fis21}, John Cremona found a pair of anti-symplectically $17$-congruent elliptic curves, which are given by the equations $A : y^2 + xy = x^3 - 8x + 27$ and $B
:y^2 + xy = x^3 + 8124402x - 11887136703,$ and with conductors $N(A)=3 \cdot 5^2 \cdot 7^2$ and $N(B)=3 \cdot 5^2 \cdot 7^2 \cdot 13.$ In particular, $11$ does not divide either conductor, and hence, $A$ and $B$ have good reduction at $11.$ The reductions of $A$ and $B$ modulo $11$ are given as $\overline{A} : y^2 + xy =x^3 +3x+5$ and $\overline{B} : y^2 +xy =x^3 +9,$ both of which are elliptic curves over $\F_{11}$. Take $\ell =17 \ne 11.$ Since $A$ and $B$ are $17$-congruent elliptic curves, there exists a $G_{\Q}$-equivariant isomorphism $\psi \colon A[17](\overline{\Q}) \rightarrow B[17](\overline{\Q})$. By restricting $\psi$ to a decomposition group at $11,$ and using the compatibility of prime-to-$11$ torsion with good reduction \cite[Propositions VII.3.1 and VII.4.1]{Sil08}, we obtain a $\mathrm{Gal}(\overline{\F}_{11}/\F_{11})$-equivariant isomorphism $\overline{\psi} \colon \overline{A}[17](\overline{\F}_{11}) \rightarrow \overline{B}[17](\overline{\F}_{11})$, which induces a quandle isomorphism $Q(\overline{A}[17](\overline{\F}_{11}), \pi_{A,17}) \cong Q(\overline{B}[17](\overline{\F}_{11}), \pi_{B,17})$. Since $\ell=17>4 \sqrt{11},$ it follows from Corollary \ref{main cor I} that $\overline{A}$ and $\overline{B}$ are isogenous over $\F_{11}.$
\end{example}

\bigskip

\noindent
\textsc{Department of Mathematics, and Institute of Pure and Applied Mathematics, Jeonbuk National University, Baekje-daero, Deokjin-gu, Jeonju-si, Jeollabuk-do, 54896, Republic of Korea}
\vskip 0.1in
\noindent
\textit{Email address:} hwangwon@jbnu.ac.kr

\end{document}